\documentclass[leqno]{article}

\usepackage{amsmath, amsthm, amsfonts, amssymb,  amscd, latexsym, stackrel, enumitem}
\usepackage[hyphens]{url}
\usepackage{hyperref}

\usepackage{xy}
\xyoption{all}

\usepackage{enumitem}
\setenumerate{label={\rm{(\arabic*)}},leftmargin=0cm,itemindent=2em,labelwidth=\itemindent,labelsep=0cm,align=left}   
\setitemize{label=$\circ$}

\makeatletter
\def\url@smallstyle{%
  \@ifundefined{selectfont}{\def\UrlFont{\sf}}{\def\UrlFont{\small\ttfamily}}}
\makeatother
\numberwithin{equation}{section}

\theoremstyle{plain}
    \newtheorem{theorem}[equation]{Theorem}
    \newtheorem{lemma}[equation]{Lemma}
    \newtheorem{corollary}[equation]{Corollary}
    \newtheorem{proposition}[equation]{Proposition}
    \newtheorem{conjecture}[equation]{Conjecture}
    
    \newtheorem*{theorem*}{Theorem}
    \newtheorem*{proposition*}{Proposition}
    \newtheorem*{corollary*}{Corollary}
    \newtheorem*{lemma*}{Lemma}
    \newtheorem*{conjecture*}{Conjecture}
    \newtheorem{definition-theorem}[equation]{Definition/Theorem}
    \newtheorem{definition-lemma}[equation]{Definition/Lemma}
\theoremstyle{definition}
    \newtheorem{definition}[equation]{Definition}
    \newtheorem{example}[equation]{Example}
    \newtheorem{examples}[equation]{Examples}

    \newtheorem{remark}[equation]{Remark}
    
    \newtheorem{remarks}[equation]{Remarks}

    \newcommand{\R}{\mathbb{R}}
    \newcommand{\C}{\mathbb{C}}
    
    \newcommand{\Z}{\mathbb{Z}}
    \newcommand{\Q}{\mathbb{Q}}

    \renewcommand{\H}{{\mathcal H}}
    \renewcommand{\phi}{\varphi}
    \renewcommand{\epsilon}{\varepsilon}
    
    \renewcommand{\injlim}{\varinjlim}
    \newcommand{\into}{\hookrightarrow}

\newcommand{\dd}{d} 
\newcommand{\reg}{\mathrm{reg}}

\newcommand{\Centre}{\mathfrak Z}

\newcommand{\argument}{\hspace{2pt}\underbar{\phantom{g}}\hspace{2pt}}
\newcommand{\functor}{\mathbb}

\DeclareMathOperator{\Hom}{Hom}
\DeclareMathOperator{\End}{End}
\DeclareMathOperator{\Trace}{Trace}

\DeclareMathOperator{\Ad}{Ad}

\DeclareMathOperator{\ind}{ind}
\DeclareMathOperator{\SL}{SL}
\DeclareMathOperator{\GL}{GL}
\DeclareMathOperator{\id}{id}
\DeclareMathOperator{\vol}{vol}

\DeclareMathOperator{\lspan}{span}

\DeclareMathOperator{\Class}{Cl}

\newcommand{\Mod}{\mathcal M}
\newcommand{\Proj}{\mathcal P}

\DeclareMathOperator{\ch}{ch}
\DeclareMathOperator{\opp}{op}

\DeclareMathOperator{\pind}{i}
\DeclareMathOperator{\pres}{r}
\DeclareMathOperator{\opind}{\overline{\hspace{.5pt}\i\hspace{.5pt}}}
\DeclareMathOperator{\opres}{\overline{r}}
\newcommand{\HH}{\operatorname{HH}}
\newcommand{\HC}{\operatorname{HC}}

\renewcommand{\O}{\mathcal{O}}
\newcommand{\h}{\operatorname{H}}
\newcommand{\SBI}{SBI}
\newcommand{\coloneq}{:=}

\begin{document}
\title{Restriction to compact subgroups in the cyclic homology of reductive $p$-adic groups}
\author{Tyrone Crisp
\thanks{\href{mailto:crisp@math.ku.dk}{\nolinkurl{crisp@math.ku.dk} }}%
}
\maketitle

\begin{abstract} 
Restriction of functions from a reductive $p$-adic group $G$ to its compact subgroups defines an operator on the Hochschild and cyclic homology of the Hecke algebra of $G$. We study the commutation relations between this operator and others coming from representation theory: Jacquet functors, idempotents in the Bernstein centre, and characters of admissible representations. 
\end{abstract}

\section{Introduction}

Let $G$ be reductive $p$-adic group. Work of Bernstein and others has led to a detailed description of the category $\Mod_f(G)$ of finitely generated smooth complex representations of $G$, in terms of parabolic induction from Levi subgroups \cite{Bernstein-Deligne}. A second approach to the representation theory of $G$, exemplified by the work of Bushnell and Kutzko (e.g., \cite{Bushnell-Kutzko_GLn}), proceeds via compact induction from compact open subgroups. The relationship between the two kinds of induction is the subject of Bushnell and Kutzko's theory of types and covers \cite{Bushnell-Kutzko_Types}. Here we study this relationship from a different point of view: that of cyclic homology.

We consider the Hochschild and cyclic homology groups $\HH_*(\Mod_f(G))$ and $\HC_*(\Mod_f(G))$ associated to $\Mod_f(G)$ as in \cite{Keller}. (The same groups may be obtained via a construction of McCarthy; see Section \ref{McCarthy_section}.) These two homology theories are related by a long exact sequence, which for the sake of brevity we will denote by $\h(\Mod_f(G))$. Results of Bernstein \cite{Bernstein-Deligne} and Keller \cite{Keller} imply that $\h(\Mod_f(G))$ may be described ``geometrically'', that is, in terms of functions on $G$: 
\begin{equation}\label{Fourier_equation} \h(\Mod_f(G)) \cong \h(\H(G)) \end{equation}
where the right-hand side is the Hochschild and cyclic homology of the Hecke algebra of $G$ (see Theorem \ref{Keller_theorem}; similar isomorphisms in degree-zero homology appear, implicitly or explicitly, in \cite{Kazhdan}, \cite{Vigneras}, \cite{Dat_K0}, and \cite{Schneider-Stuhler}, among others).     

The cyclic homology groups of $\H(G)$ have been studied in \cite{Julg-Valette1}, \cite{Julg-Valette2}, \cite{Blanc-Brylinski}, \cite{Higson-Nistor}, and  \cite{Schneider}, where a central role is played by a certain idempotent operator $1_{G_c}$ on $\h(\H(G))$, defined by restricting functions from $G$ to the union of its compact subgroups (see Examples \ref{Keller_examples}(6) for the precise definition). In degree zero, the isomorphism \eqref{Fourier_equation} restricts to
\[  \sum_{\substack{K\subset G \\ \text{compact}\\\text{open}}} \ind_K^G \HH_0(\Mod_f(K)) \cong 1_{G_c}\HH_0(\H(G)),\] where $\ind_K^G:\Mod_f(K)\to \Mod_f(G)$ is the functor of compact induction. In higher degrees, Higson and Nistor \cite{Higson-Nistor} and Schneider \cite{Schneider} have given a description of the image of $1_{G_c}$ in terms of \emph{chamber homology}, which combines the groups $\h(\Mod_f(K))$ and the combinatorics of the Bruhat-Tits building of $G$. 

Motivated by this close connection between the ``compact-restriction'' operator $1_{G_c}$ and the compact-induction functors $\ind_K^G$, we study the commutation relations between $1_{G_c}$ and other representation-theoretic operators: parabolic induction, Jacquet restriction, idempotents in the Bernstein centre, and characters of admissible families of representations. Our main results are summarised below. These results are applied in \cite{Crisp_Bernstein} and \cite{Crisp_Parahoric}.

The functors of parabolic induction and Jacquet restriction with respect to a Levi subgroup $M\subset G$ induce, via \eqref{Fourier_equation}, maps in homology:
\[ \xymatrix@1@C=50pt{ \h(\H(G)) \ar@<0.5ex>[r]^{\pres} & \h(\H(M)) \ar@<0.5ex>[l]^{\pind}}\] 
In Proposition \ref{Jacquet_proposition} we compute the Jacquet restriction map $\pres:\h(\H(G))\to \h(\H(M))$, and show that it is equal to one defined by Nistor in \cite{Nistor}. Nistor suggested that his map be considered an analogue of parabolic induction, and our computation makes this analogy precise. One consequence is that Jacquet restriction commutes with compact induction:
$ \pres 1_{G_c} = 1_{M_c} \pres$. 
In degree-zero homology, this has previously been observed by Dat \cite[Lemme 2.6]{Dat_Idempotents}.
 
Each idempotent $E$ in the Bernstein centre $\Centre(G)$ induces an idempotent endomorphism of $\h(\H(G))$. Results of Higson and Nistor \cite{Higson-Nistor} and Schneider \cite{Schneider} imply that for every endomorphism $T$ of $\h(\H(G))$, the commutator $[T,1_{G_c}]$ is nilpotent of order at most $3$ (see Lemma \ref{SBI_lemma}). Dat has shown that for the idempotents in the Bernstein centre one in fact has $[E,1_{G_c}]=0$ as operators on $\HH_0(\H(G))$. We conjecture that the same holds on all of $\h(\H(G))$, and we prove this conjecture for $G=\SL_2(F)$ (Theorem \ref{SL2_commutation_theorem}; the same argument shows that $[E,1_{G_c}]=0$ on $\HH_n(\H(G))$ for $G$ a split reductive group of rank $n$). 

Dat has also shown, using a formula of Clozel \cite[Proposition 1]{Clozel}, that parabolic induction does not commute with compact restriction in the degree-zero homology of $\SL_2(F)$ \cite[Remarque, p.77]{Dat_Idempotents}. By extending Clozel's formula  to higher homology (in the special case of $\SL_2(F)$), we prove that parabolic induction does commute with compact restriction in higher degrees (Theorem \ref{SL2_Clozel_theorem} and Corollary \ref{SL2_induction_corollary}). Moreover, we show that the failure to commute in degree zero is confined to a single Bernstein component---the unramified principal series---and we derive an explicit formula for the commutator in terms of the Iwahori-Hecke algebra. In particular, we show that this commutator is a rank-one map (Proposition \ref{Iwahori_commutator_proposition}). We conjecture that Clozel's formula is valid in higher homology for all reductive $p$-adic groups. In Section \ref{Weyl_section} we prove the analogue of this conjecture for affine Weyl groups (Proposition \ref{Weyl_Clozel_proposition}).  

Each admissible representation $\pi$ of $G$ determines a map $\HH_0(\H(G))\to \C$, the character of $\pi$. This construction may be extended to families of representations: if $X$ is a complex affine variety, and $\pi$ is an admissible algebraic family of representations of $G$ parametrised by $X$ (in the sense of \cite[1.16]{Bernstein-Deligne}), then the functor $\Hom_G(\argument,\pi):\Mod_f(G)\to \Mod_f(\mathcal O(X))$ induces a map $\ch_\pi:\h(\H(G))\to \h(\mathcal O(X))$. For example, if $\sigma$ is an irreducible supercuspidal representation of a Levi subgroup $M\subseteq G$, and $\Psi$ is the complex torus of unramified characters of $M$, then the parabolically induced representation $\pi= \pind\left(\O(\Psi)\otimes\sigma\right)$ is an admissible family over $\Psi$. The compact-restriction operator $1_{\Lambda_c}$ for the lattice $\Lambda=\Hom(\Psi,\C^\times)$ acts on the homology of $\mathcal O(\Psi)$, and we prove in Proposition \ref{character_proposition} that 
\begin{equation*}\label{character_commutation} 1_{\Lambda_c}\ch_\pi   = \ch_\pi 1_{G_c} \end{equation*}
as maps $\HH_0(\H(G))\to \HH_0(\mathcal O(\Psi))$. If Clozel's formula holds in the higher-degree homology of $M$---for example, if $M$ is a torus---then the above equality is valid on all of $\h(\H(G))$.

\bigskip

This research was partially supported under NSF grant DMS-1101382, and by the Danish National Research Foundation through the Centre for Symmetry and Deformation (DNRF92). Some of the results have previously appeared in the author's Ph.D. thesis, written at The Pennsylvania State University under the direction of Nigel Higson.

\section{Functoriality of cyclic homology}\label{McCarthy_section}

Let $G$ be a reductive $p$-adic group: i.e., the group of $F$-points of a connected reductive group defined over $F$, where $F$ is a finite extension of $\Q_p$. The Hecke algebra $\H(G)$ of locally constant, compactly supported functions $G\to\C$ is an associative algebra under convolution with respect to a choice of Haar measure. We consider the Hochschild and cyclic homology groups $\HH_*(\H(G))$ and $\HC_*(\H(G))$ of this algebra. Our basic reference for cyclic homology is \cite{Loday}. 

Let $\Mod_{f}(G)$  denote the category of finitely generated, smooth representations of $G$, viewed as an exact category enriched over $\C$-vector spaces; let $\HH_*(\Mod_{f}(G))$ and $\HC_*(\Mod_{f}(G))$ denote the Hochschild and cyclic homology groups associated to this category by Keller \cite{Keller}.

The Hochschild and cyclic homology groups of an object $C$ are related by an exact sequence
\[ \ldots \to \HC_{n+1}(C) \xrightarrow{S} \HC_{n-1}(C) \xrightarrow{B} \HH_n(C) \xrightarrow{I}  \HC_{n}(C) \xrightarrow{S} \HC_{n-2}(C) \to \ldots\]
It will be convenient to use the notation $\h(C)$ to refer to this sequence, so that for example ``$f:\h(C)\to \h(C')$'' means that $f$ is a pair of graded linear maps $\HH_*(C)\to \HH_*(C')$ and $\HC_*(C)\to \HC_*(C')$ that commute with the maps $S$, $B$ and $I$. 

The following is a consequence of results of Bernstein and Keller:

\begin{theorem}\label{Keller_theorem} For every reductive $p$-adic group $G$ one has an isomorphism
$\h(\H(G))\cong \h(\Mod_{f}(G))$.
\end{theorem}

\begin{proof} According to the Bernstein decomposition \cite{Bernstein-Deligne}, $\H(G)$ is a direct sum of two-sided ideals, each of which is Morita equivalent to a unital, Noetherian algebra of finite global dimension. Keller has shown that for each such algebra $A$ one has $\h(A)\cong \h(\Mod_f(A))$ \cite[1.6]{Keller}. The functor $\h$ commutes with direct sums and is Morita invariant, so the result follows.
\end{proof}

The functoriality of Keller's construction \cite[1.14]{Keller} then gives:

\begin{corollary}\label{Keller_corollary} 
Let $G$ and $G'$ be reductive $p$-adic groups. Each derivable (e.g., exact) $\C$-linear functor $\Mod_{f}(G)\to \Mod_{f}(G')$ induces a canonical map $\h(\H(G))\to \h(\H(G'))$, such that composition of functors corresponds to composition of maps. If $\functor{E} \to \functor{F}\to \functor{G}$ is a short exact sequence of functors, then $\functor{F}=\functor{E}+\functor{G}$ as maps $\h(\H(G))\to \h(\H(G'))$.\hfill\qed
\end{corollary}

Theorem \ref{Keller_theorem} can also be applied in the other direction. Let $\Class^\infty(G)$ denote the space of locally constant, conjugation-invariant functions $G\to \C$. This is an algebra, under pointwise multiplication, and Blanc and Brylinski have shown that $\h(\H(G))$ is in a natural way a $\Class^\infty(G)$-module. (An explicit formula for the module structure is recalled in \eqref{Class_cyclic_equation}, below.) 

\begin{corollary}\label{Keller_corollary2}
$\h(\Mod_f(G))$ is a module over $\Class^\infty(G)$. \hfill\qed
\end{corollary}

\begin{examples}\phantomsection\label{Keller_examples}
\begin{enumerate}
\item Automorphisms: Each algebra automorphism $\alpha$ of $\H(G)$ gives rise to an exact functor---``twist by $\alpha$''---on $\Mod_{f}(G)$. The induced automorphism of $\h(\H(G))$ is the same as the one induced by $\alpha$ as an algebra automorphism. 

\item Central idempotents: Let $E$ be an idempotent in the Bernstein centre $\Centre(G)$ (see \cite{Bernstein-Deligne}). The exact functor $V\mapsto EV$ on $\Mod_{f}(G)$ induces an endomorphism of $\h(\H(G))$, equal to the one induced by $E$ as an endomorphism of the algebra $\H(G)$.

\item Jacquet functors: Let $M$ be a Levi component of a parabolic subgroup $P$ of $G$, and consider the functors $\pind_M^G$ and $\pres^G_M$ of normalised parabolic induction and Jacquet restriction along $P$ \cite[2.5]{Bernstein-Deligne}.
Each of these functors is exact, and preserves the property of being finitely generated \cite[Section 3]{Bernstein-Deligne}, so they induce canonical maps in Hochschild and cyclic homology.
Recall the geometric lemma \cite[2.12]{Bernstein-Zelevinsky_InducedI} of Bernstein and Zelevinsky: given two parabolic subgroups in $P,Q\subset G$, with Levi factors $M\subset P$ and $L\subset Q$, the composite functor $\pres^G_L\pind_M^G$ admits a filtration with subquotients of the form $\pind_{L_w}^L \Ad_w \pres^M_{M_w}$, where $w$ ranges over a set $W$ of coset representatives for $Q\backslash G/P$, and $M_w,L_w$ are Levi factors of certain parabolic subgroups of $M$ and $L$, respectively. This filtration becomes a sum in homology:
\begin{equation*}\label{geometrical_lemma} 
\pres^G_L \pind_M^G = \sum_{w\in W} \pind_{L_w}^L \Ad_w \pres^M_{M_w} : \h(\H(M))\to \h(\H(L)).
\end{equation*}

\item Characters: Let $X$ be a complex affine variety, with coordinate algebra $\O(X)$. Recall from \cite[1.16]{Bernstein-Deligne} that an \emph{algebraic family of representations} of $G$ over $X$, also called a $(G,X)$-module, is an $\H(G)$--$\O(X)$ bimodule that is flat over $\O(X)$. Such a module $V$ is \emph{admissible} if for each compact open subgroup $K\subset G$, the space $V^K$ of $K$-invariants is finitely generated over $\O(X)$.  
If $V$ is admissible, and $M\in \Mod_f(G)$, then $\Hom_G(M,V)$ is finitely generated over $\O(X)$. The resulting functor $\Hom_G(\argument, V):\Mod_{f}(G)\to \Mod_{f}(\O(X))$ is exact on the subcategory $\Proj_{f}(G)$ of projectives in $\Mod_f(G)$, and is therefore derivable. So this functor induces a map $\ch_V:\h(\H(G))\to \h(\O(X))$. The proof of Proposition \ref{McCarthy_proposition} (below) will make it clear that in degree zero this map is given by
\[\ch_V:\H(G)/[\H(G),\H(G)] \to \O(X),\quad \ch_V(f)(x)=\Trace\left( V_x \xrightarrow{f} V_x\right),\] the trace of $f\in \H(G)$ as an operator on the fibre $V_x$ over $x\in X$. It therefore seems appropriate to call $\ch_V$ the \emph{character} of $V$. If $V_1\to V_2\to V_3$ is a short exact sequence of admissible $(G,X)$-modules, then the sequence of functors $\Hom_G(\argument,V_1)\to \Hom_G(\argument,V_2)\to \Hom_G(\argument,V_3)$ is exact on $\Proj_f(G)$, and we therefore have $\ch_{V_2}=\ch_{V_1}+\ch_{V_3}$.
\item Compact induction: For each compact open subgroup $K\subset G$ one has a functor $\ind_K^G:\Mod_f(K)\to \Mod_f(G)$ of compact induction \cite[I.3.2]{Bernstein_notes}. The corresponding map $\h(\H(K))\to \h(\H(G))$ is equal to the one induced by the inclusion of algebras $\H(K)\into \H(G)$.
\item Compact restriction: Let $G_c$ denote the union of the compact subgroups of $G$. This is an open, closed, and conjugation-invariant subset of $G$, so its characteristic function $1_{G_c}$ lies in $\Class^\infty(G)$. The corresponding idempotent endomorphism of $\h(\H(G))$ (and $\h(\Mod_f(G))$) will be called \emph{compact restriction}. The operator $1_{G_c}$ is related to the compact-induction maps $\ind_K^G$, as explained in the introduction.
\end{enumerate}
\end{examples}

Using a construction due to McCarthy \cite{McCarthy}, we shall now give an explicit description of the groups $\h(\Mod_f(G))$ and the isomorphism $\h(\H(G))\cong \h(\Mod_f(G))$. This description is useful for computations involving the map in cyclic homology induced by a functor $\functor{F}:\Mod_f(G)\to \Mod_f(G')$, in cases where $\functor{F}$ restricts to a functor between the subcategories of finitely generated projectives. Note that the functors in Examples \ref{Keller_examples}(1)--(5) all have this property. (For parabolic induction, this is a corollary of Bernstein's second adjoint theorem \cite{Bernstein_Second_Adjoint}.)

%
%
%

Before stating the results, let us establish some notation. An (associative, complex) algebra $A$ is \emph{locally unital} if for every finite subset $S\subset A$, there exists an idempotent $e\in A$ such that $es=se=s$ for each $s\in S$. All (left) modules $V$ over $A$ will be assumed to be \emph{nondegenerate}, i.e., to satisfy if $V=AV$. Our main example is, of course, the Hecke algebra $\H(G)$, whose nondegenerate modules are precisely the smooth representations of $G$. 

To each locally unital algebra $A$ we associate a precyclic module $C(A)$, and Hochschild and cyclic homology groups $\HH_*(A)$ and $\HC_*(A)$, as usual (see \cite{Loday}; note that $A$ is $H$-unital, so the ``naive'' definitions suffice). For $A=\H(G)$, one has $C_n(\H(G))\cong \H(G^{n+1})$; explicit formulas for the structure maps in this picture are given in \cite{Blanc-Brylinski}. The algebra $\Class^\infty(G)$ acts on $\H(G^{n+1})$ according to the formula
\begin{equation}\label{Class_cyclic_equation} (F f)(g_0,\ldots,g_n) = F(g_0\cdots g_n) f(g_0,\ldots,g_n)\end{equation} for $F\in\Class^\infty(G)$ and $f\in \H(G^{n+1})$.

Now let $\mathcal A$ be a small category enriched over $\C$-vector spaces. Following Mitchell \cite[\S 17]{Mitchell} and McCarthy \cite[\S2.1]{McCarthy}, we define a precyclic module by letting $C_n(\mathcal A)$ be the vector space
\[ \bigoplus_{(A_0,\ldots,A_n)\in \mathcal A^{n+1}} \Hom(A_0,A_1)\otimes_{\C}\Hom(A_1,A_2) \otimes_{\C}\cdots\otimes_{\C}\Hom(A_n,A_0).\] 
The structure maps are defined by
\[ d_i(f_0\otimes\cdots\otimes f_n)=\begin{cases} f_0\otimes\cdots\otimes f_{i+1} f_{i}\otimes\cdots\otimes f_n&\text{if }0\leq i<n,\\
f_0 f_n\otimes f_1\otimes\cdots\otimes f_{n-1}&\text{if }i=n\end{cases}\]
\[t(f_0\otimes\cdots\otimes f_n)=f_n\otimes f_0\otimes\cdots\otimes f_{n-1}.\]
The associated Hochschild and cyclic homology groups will be denoted $\HH_*^s(\mathcal A)$ and $\HC_*^s(\mathcal A)$. We let $\h^s(\mathcal A)$ denote the $\SBI$ long exact sequence. The superscripts $s$ (for ``split'') are used to distinguish this construction from the more elaborate ones in \cite{McCarthy} and \cite{Keller}. Each functor $\mathcal A\to \mathcal A'$ induces a map of precyclic modules $C(\mathcal A)\to C(\mathcal A')$, and thus maps in homology $\h^s(\mathcal A)\to \h^s(\mathcal A')$.

\begin{example}\label{McCarthy_example} Let $A$ be a unital algebra, and let $(\star,A^{\opp})$ be the category with one object $\star$, having $\End(\star)=A^{\opp}$ (the algebra opposite to $A$). Then $C(\star,A^{\opp})\cong C(A)$, the standard precyclic module associated to $A$. Let $\Proj_f(A)$ denote the category of finitely generated, projective left modules over $A$. This category is not small, but it admits a small skeleton, and we define $C(\Proj_f(A))$ in terms of such a skeleton. There is a covariant inclusion $\functor{I}:(\star,A^{\opp})\to \Proj_f(A)$, sending $\star$ to $A$, and $a\in A^{\opp}$ to the endomorphism $a'\mapsto a'a$ of $A$. McCarthy has shown that this inclusion induces isomorphisms in Hochschild and cyclic homology \cite[Proposition 2.4.3]{McCarthy}: the inverse is given by combining the trace maps $\h(\End P)\to \h(A)$ associated to the various $P\in \Proj_f(A)$.
\end{example}  

We will extend McCarthy's result to locally unital algebras. The existence of an isomorphism $\h(A)\cong \h^s(\Proj_f(A))$ in this case is easily established, by considering the embedding of $A$ into its minimal unitalisation $A^+$. We will later need an explicit isomorphism, which we now construct. 

For each idempotent $e\in A$, let $\Proj_f(A,e)\coloneq \{ P\in \Proj_f(A)\ |\ P=AeP\}$, a full subcategory of $\Proj_f(A)$. Define a functor 
$\functor{F}_e:(\star,eAe^{\opp})\to \Proj_f(A,e)$ by setting $\functor{F}_e(\star)=Ae$, and letting $\functor{F}_e(a)\in \End_A(Ae)$ be the operator of right-multiplication by $a$, for each $a\in eAe$. The set of idempotents in $A$ is directed according to the partial order $e\leq f \iff ef=fe=e$, and the inclusion maps $eAe\into A$ and $\Proj_f(A,e)\into \Proj_f(A)$ induce isomorphisms $C(A)\cong \injlim C(eAe)$ and $C(\Proj_f(A))\cong \injlim C(\Proj_f(A,e))$.

\begin{proposition}\label{McCarthy_proposition} Let $A$ be a locally unital algebra, and $\Proj_f(A)$ the category of finitely generated projective $A$-modules. There are isomorphisms
\[ \h(A)\xrightarrow{\cong} \injlim \h^s(\star,eAe^{\opp}) \xrightarrow[\injlim \functor{F}_e]{\cong} \injlim \h^s(\Proj_f(A,e)) \xrightarrow{\cong} \h^s(\Proj_f(A))\] in Hochschild and cyclic homology, compatible with $\SBI$ sequences.
\end{proposition}
For the degree-zero homology of the Hecke algebra of a reductive $p$-adic group, this result was stated in \cite{Dat_K0}.

\begin{proof} We have already observed that the first and third isomorphisms hold at the level of precyclic modules. For each idempotent $e$, the functor $\functor{F}_e$ is the composition
\[ (\star,eAe^{\opp})\xrightarrow{\functor{I}} \Proj_f(eAe) \xrightarrow{\functor{G}} \Proj_f(A,e),\]
where $\functor{I}$ is as in Example \ref{McCarthy_example}, and $\functor{G}(P)=Ae\otimes_{eAe}P$. McCarthy proved that $\functor{I}$ induces an isomorphism in homology. The functor $\functor{G}$ is an equivalence of categories, so $\functor{F}_e$ induces an isomorphism in homology.
It remains to show, for idempotents $e\leq f$, that the diagram
\[ \xymatrix{
\h^s(\star,eAe^{\opp}) \ar[r]^-{\functor{F}_e} \ar[d] & \h^s(\Proj_f(A,e)) \ar[d] \\
\h^s(\star,fAf^{\opp}) \ar[r]^-{\functor{F}_f} & \h^s(\Proj_f(A,f))
}
\]
commutes. We do this by constructing a special homotopy (as defined in \cite[Definition 2.3.2]{McCarthy}) between the two compositions. 

Let $(\mathcal I,eAe^{\opp})$ be the category with two objects, $0$ and $1$, with $\Hom(i,j)=eAe^{\opp}$ for each $i$ and $j$, and with composition of morphisms given by multiplication in $eAe^{\opp}$. For each $a\in eAe$, we write $a_{i,j}$ for the corresponding element of $\Hom(i,j)$. We consider the inclusion functors 
\[\epsilon_i:(\star,eAe^{\opp})\to (\mathcal I,eAe^{\opp}),\quad \epsilon_i(\star)=i,\quad \epsilon_i(a)=a_{i,i}\quad (i=0,1,\ a\in eAe).\]

Now define a functor $\functor{J}:(\mathcal I,eAe^{\opp}) \to \Proj_f(A,f)$ as follows. On objects, $\functor{J}(0)\coloneq Af $ and $\functor{J}(1)\coloneq Ae$. On morphisms, for each $a\in eAe$ we let $\functor{J}(a_{i,j}):\functor{J}(i)\to\functor{J}(j)$ be right-multiplication by $a$. The diagram
\[ 
\xymatrix{
(\star,eAe^{\opp}) \ar[r]^-{\epsilon_0} \ar[d]  & (\mathcal I,eAe^{\opp}) \ar[d]^-{\functor{J}} & (\star,eAe^{\opp}) \ar[l]_-{\epsilon_1} \ar[d]^-{\functor{F}_e} \\
(\star,fAf^{\opp})\ar[r]^-{\functor{F}_f} & \Proj_f(A,f) & \Proj_f(A,e) \ar[l]
}
\]
commutes, and so $\functor{J}$ implements a special homotopy between the two compositions in 
\[ \xymatrix{
(*,eAe^{\opp}) \ar[r]^-{\functor{F}_e} \ar[d] & \Proj_f(eAe) \ar[d] \\
(*,fAf^{\opp}) \ar[r]^-{\functor{F}_f} & \Proj_f(fAf)
}
\]
The induced diagram in homology therefore commutes, by \cite[Proposition 2.3.3]{McCarthy}. Thus the isomorphisms $\functor{F}_e$ assemble into an isomorphism of direct limits.  
\end{proof}

The relevance of Proposition \ref{McCarthy_proposition} to reductive $p$-adic groups is explained by the following proposition, which follows immediately from \cite[Theorem 1.5]{Keller} and the Bernstein decomposition as in Theorem \ref{Keller_theorem}.

\begin{proposition}\label{McCarthy-Keller_proposition}
For each reductive $p$-adic group $G$, there is an isomorphism $\h^s(\Proj_f(G)) \xrightarrow{\cong} \h(\Mod_f(G))$ which is related to the isomorphisms of Theorem \ref{Keller_theorem} and Proposition \ref{McCarthy_proposition} by a commuting diagram
 \[\xymatrix{ \h(\H(G)) \ar[r]^-{\cong} \ar[dr]_-{\cong} & \h^s(\Proj_f(G)) \ar[d]^-{\cong} \\ & \h(\Mod_f(G)).}\]
If $\functor{F}:\Mod_f(G)\to \Mod_f(G')$ is a functor which restricts to an exact functor $\Proj_f(G)\to \Proj_f(G')$, then the diagram
\[ \xymatrix{ \h^s(\Proj_f(G)) \ar[d]_-{\functor{F}} \ar[r]^-{\cong} & \h(\Mod_f(G)) \ar[d]^-{\functor{F}} \\ 
 \h^s(\Proj_f(G')) \ar[r]^-{\cong} & \h(\Mod_f(G)) }
\]
is commutative.\hfill\qed
\end{proposition}

%

\section{Comparing geometric and spectral operators}\label{operators_section}

Throughout this section we consider the Hecke algebra $\H(G)$ of a reductive group $G$ over a $p$-adic field $F$. We refer to \cite{Bernstein_notes} and \cite{Renard} for the general theory and terminology.

The operator $1_{G_c}:\h(\H(G))\to \h(\H(G))$ of compact restriction was defined in Examples \ref{Keller_examples}(6). The image of this idempotent on Hochschild and cyclic homology, and on the $\SBI$ long exact sequence, will be denoted by $\HH_*(\H(G))_c$, $\HC_*(\H(G))_c$ and $\h (\H(G))_c$, respectively. The image of the complementary idempotent $1_{G_{nc}}=1-1_{G_c}$ will be denoted by $\HH_*(\H(G))_{nc}$, etc. 


Higson and Nistor \cite{Higson-Nistor} and Schneider \cite{Schneider} have shown that the compact-restriction operator interacts with the maps in the $\SBI$ sequence according to the formulas
\begin{equation}\label{SBI_relations} 1_{G_c} B= B 1_{G_c} = 0 \qquad\text{and}\qquad 1_{G_c} S= S {1_{G_c}} = S.
\end{equation}
It follows immediately from these formulas that $1_{G_c}$ acts as the identity on the periodic cyclic homology of $\H(G)$; see \cite{Julg-Valette1}, \cite{Julg-Valette2}, \cite{Blanc-Brylinski}, \cite{Higson-Nistor} and \cite{Schneider} for details and historical background. The following lemma lists some more consequences of the relations \eqref{SBI_relations}. Recall that by a map $\h(\H(G))\to \h(\H(G'))$ we mean a pair of graded linear maps $\HH_*(\H(G))\to \HH_*(\H(G'))$ and $\HC_*(\H(G))\to \HC_*(\H(G))$ commuting with the maps $S$, $B$ and $I$.

\begin{lemma}\label{SBI_lemma}
Let $G$, $G'$ and $G''$ be reductive $p$-adic groups, and suppose $T:\h (\H(G)) \to \h (\H(G'))$ and $T':\h (\H(G')) \to \h (\H(G''))$ are maps of $\SBI$ sequences.
\begin{enumerate}
\item The inclusion of cochain complexes
\[ (\HH_*(\H(G))_c, 0)\into (\HH_*(\H(G)), B I)\]
is a quasi-isomorphism.
\item $T\left(\HC_*(\H(G))_c\right)\subseteq \HC_*(\H(G'))_c$.
\item $(T' T)_c = T'_c T_c$, where $T_c=1_{G'_c} T 1_{G_c}$, and so on. 
\item $[T,1_{G_c}]^2=0$ on $\HC_*(\H(G))$, and $[T,1_{G_c}]^3=0$ on $\h (\H(G))$.
\end{enumerate}
\end{lemma}

\begin{proof} Parts (1), (2) and (3) follow from the relations \eqref{SBI_relations} by routine diagram-chases in the $\SBI$ sequence. In part (4), one has by (2) and (3) that
\[ [T,1_{G_c}]^2 = TT_c - T 1_{G_c} T - T_c^2 + T_c T,\] which vanishes on $\HC_*(\H(G))$ because $T1_{G_c}=T_c$. Multiplying once more by $[T,1_{G_c}]$ and applying (3) several times, one finds that $[T,1_{G_c}]^3=0$ on $\h(\H(G))$. 
 \end{proof}

\begin{corollary} 
The operators $E_c\coloneq 1_{G_c} E 1_{G_c}$, for $E$ ranging over the set of minimal idempotents in $\Centre(G)$, form a family of pairwise-orthogonal idempotents on $\h (\H(G))$, with $\sum E_c = 1_{G_c}$.\hfill\qed
\end{corollary}

One may use the idempotents $E_c$ to decompose $\HH_*(\H(G))_c$ as a direct sum of ``Bernstein components''. Identifying $\HH_*(\H(G))_c$ with chamber homology, one finds that these components coincide with the ones studied by Baum, Higson and Plymen in \cite{Baum-Higson-Plymen}. This decomposition is studied further in \cite{Crisp_Bernstein}.

The remainder of this section studies the commutation relations between $1_{G_c}$ and other operators coming from representation theory. We begin with the very simple case where $G$ is a torus.

\subsection*{Tori}

First let $\Lambda$ be a free abelian group of finite rank, and let $\Psi$ be the complex torus $\Hom(\Lambda,\C^\times)$. The Fourier transform gives an isomorphism of algebras $\H(\Lambda)\cong \mathcal O(\Psi)$, and a corresponding isomorphism in Hochschild homology $\HH_*(\H(\Lambda))\cong \HH_*(\mathcal O(\Psi))\cong \Omega^*(\Psi)$, the second isomorphism being the Hochschild-Kostant-Rosenberg theorem \cite{HKR}. The complex $\Omega^*(\Psi)$ decomposes into eigenspaces for the translation action of $\Psi$, and we let $\Pi_0$ denote the projection onto the subspace of invariant forms. 

A straightforward computation shows that the diagram 
\begin{equation}\label{harmonic_equation}
\xymatrix{ \HH_*(\H(\Lambda)) \ar[r]^-{\cong} \ar[d]_-{1_{\Lambda_{c}}} & \Omega^*(\Psi) \ar[d]^-{\Pi_0} \\
\HH_*(\H(\Lambda)) \ar[r]^-{\cong} & \Omega^*(\Psi)
}
\end{equation}
commutes. So Lemma \ref{SBI_lemma}(1) in this case is the well-known quasi-isomorphism between the de Rham complex of $\Psi$ and the subcomplex of invariant forms.

Now let $T$ be a torus over the $p$-adic field $F$, with maximal compact subgroup $T^\circ$. The quotient $\Lambda\coloneq T/T^\circ$ is a free abelian group, and the dual group $\Psi=\Hom(\Lambda,\C^\times)$ of unramified characters is a complex torus. The Hochschild homology of $\H(T)$ is given by 
\begin{equation}\label{torus_equation} \HH_*(\H(T))\cong \bigoplus_{\widehat{T^\circ}} \HH_*(\H(\Lambda)) \cong \bigoplus_{\widehat{T^\circ}} \Omega^*(\Psi),\end{equation}
where $\widehat{T^\circ}$ is the set of smooth homomorphisms $T^\circ\to \C^\times$. See \cite[Proposition 4.7]{Blanc-Brylinski} for details, and for the corresponding description of cyclic homology. 

The minimal idempotents $E\in\Centre(T)$ are parametrised by $\widehat{T^\circ}$; the idempotent $E_\sigma$ associated to $\sigma\in\widehat{T^\circ}$ acts by projecting $\HH_*(\H(T))$ onto the corresponding copy of $\Omega^*(\Psi)$ in \eqref{torus_equation}. The compact-restriction operator $1_{T_c}$, on the other hand, acts on each summand $\Omega^*(\Psi_T)$, as described in \eqref{harmonic_equation}. 

\subsection*{Compact restriction and Jacquet restriction}

Let $P=MN$ be a parabolic subgroup of the reductive group $G$, and let $\pres=\pres^G_M$ be the corresponding Jacquet restriction functor. We choose and fix a compact open subgroup $K\subset G$ such that $G=NMK$. Writing $K_M$ for $K\cap M$, and $K_N$ for $K\cap N$, we normalise the Haar measures on $M$, $N$ and $K$ so that $\vol_K(K)=\vol_M(K_M)=\vol_N(K_N)=1$. With these choices, we have
\begin{equation}\label{KMN_measure_equation} \int_G f(g)\, \dd g = \int_K \int_M \int_N f(nmk) \delta_P(m)\, \dd n\, \dd m\, \dd k\end{equation}
for each $f\in \H(G)$. Here $\delta_P$ denotes the modular function of $P$, characterised by $\dd(pq)=\delta(q)\dd p$ for any left Haar measure $\dd p$ on $P$.

\begin{proposition}\label{Jacquet_proposition} The Jacquet-restriction map $\pres :\h (\H(G))\to \h (\H(M))$ is the one induced by the map of precyclic modules
$\Phi:\H(G^{q+1})\to \H(M^{q+1})$,
\begin{align*}
\Phi(f)&(m_0,\ldots,m_q) =\\
& \delta_P^{1/2}(m_0\cdots m_q) \int_{K^{q+1}}\int_{N^{q+1}} f(k_0^{-1} n_0 m_0 k_1, \ldots, k_q^{-1} n_q m_q k_0) \, \dd n\, \dd k
\end{align*}
where $n=(n_0,\ldots,n_q)$ and $k=(k_0,\ldots,k_q)$.
\end{proposition}

\begin{remarks}\label{Jacquet_remarks}\begin{enumerate}
\item The  map $\Phi$ appears in \cite{Nistor}, under the name $\inf_M^P\ind_P^G$. Nistor proposes that this map should be considered an analogue, in Hochschild homology, of the parabolic induction functor. Proposition \ref{Jacquet_proposition} makes this analogy precise.
\item Van Dijk proves in \cite{van_Dijk} that for each admissible representation $V$ of $M$, one has $\ch_{\pind_M^G V}=\ch_V \Phi$ as maps $\HH_0(\H(G))\to \C$. Since characters separate points in degree-zero Hochschild homology \cite[Theorem 0]{Kazhdan}, van Dijk's result establishes Proposition \ref{Jacquet_proposition} in degree zero. This was observed by Dat in \cite[Proposition 1.10]{Dat_K0}.
\end{enumerate}
\end{remarks}

\begin{proof}[Proof of Proposition \ref{Jacquet_proposition}]
Let $f_0\otimes\cdots\otimes f_q\in C_q(\H(G))$, and find a compact, open, normal subgroup $J\subseteq K$ such that each $f_i$ is $J$-bi-invariant. Right-convolution by $f_i$ defines an endomorphism of the finitely generated, projective, left $G$-module $C_c^\infty(G/J)$. Integration over $N$ gives an $M$-equivariant isomorphism $\pres\left(C_c^\infty(G/J)\right)\xrightarrow{\cong} C_c^\infty(N\backslash G/J)$, with $M$ acting on $C_c^\infty(N\backslash G/J)$ by
\[ (mf)(NgJ)\coloneq \delta_P^{-1/2}(m) f(Nm^{-1}gJ).\] To lighten the notation, we will write $R\coloneq C_c^\infty(N\backslash G/J)$.

Applying the functor $\pres$ to the endomorphisms $f_i$ gives
\[ \pres(f_i)\in \End_M(R),\qquad \pres(f_i)(f)(NgJ)=\int_G f(NghJ)f_i(h^{-1})\, \dd h.\]
Our goal is to show that the isomorphism $\h(\Proj_f(M))\xrightarrow{\phi} \h(\H(M))$ of Proposition \ref{McCarthy_proposition} sends the class of $\pres(f_0)\otimes\cdots\otimes \pres(f_q)\in C_q(\Proj_f(M))$ to the class of $\Phi(f_0\otimes\cdots\otimes f_q)\in C_q(\H(M))$: i.e., that the diagram
\[
 \xymatrix{
 \h(\Proj_f(G)) \ar[d]_-{\pres} \ar[r]^-{\cong} & \h(\H(G)) \ar[d]^-{\Phi} \\
 \h(\Proj_f(M)) \ar[r]^-{\phi} & \h(\H(M))
 }
\]
commutes.

Adapting McCarthy's construction of the isomorphism $\phi$ to the locally unital setting, one finds that $\phi(\pres(f_0)\otimes\cdots\otimes \pres(f_q))$ may be described as follows. Find an idempotent $e\in \H(M)$ such that $R=\H(M)eR$. Then find $\alpha_i\in\Hom_M(R,\H(M)e)$ and $\beta_i\in\Hom_M(\H(M)e,R)$ (for $i$ ranging over some finite set), such that $\sum_i \beta_i\alpha_i=1_{R}$; this is possible because $R$ is finitely generated and projective. We then have
\begin{equation}\label{Jacquet_proof_eq} \begin{aligned} &\phi\left( \pres(f_0)\otimes\cdots\otimes\pres(f_q)\right) = \\
&\sum_{(i_0,\ldots,i_q)} \alpha_{i_1}\pres(f_0)\beta_{i_0}(e)\otimes  \cdots \otimes \alpha_{i_q}\pres(f_{q-1})\beta_{i_{q-1}}(e)\otimes \alpha_{i_0}\pres(f_q)\beta_{i_q}(e),\end{aligned}\end{equation}
the equality holding in $\h(\H(M))$. (See \cite[Proof of 2.4.3]{McCarthy}.)

The Iwasawa decomposition $G=NMK$ ensures that the module $R$ is generated over $\H(M)$ by the finite-dimensional subspace $S\coloneq C_c^\infty(K_N\backslash K/J)$ of functions supported on $NK$. Choose a compact, open subgroup $L\subset K_M$ which acts trivially on $S$, and let $e=e_L\in \H(M)$ be the normalised characteristic function of $L$. It is immediate from our choices that $R=\H(M)eR$.

Choose representatives $\kappa_1,\ldots,\kappa_d\in K$ for the double-coset space $K_N\backslash K/J$, and for each $i$ let $\chi_i\in S$ denote the characteristic function of $K_N \kappa_i J$. Define $M$-equivariant maps $\alpha_i$ and $\beta_i$ by
\[ \alpha_i: R\to \H(M)e,\qquad \alpha_i f(m)\coloneq \delta_P^{1/2}(m) f(Nm\kappa_i J),\]
\[ \beta_i: \H(M)e\to R,\qquad \beta_i f(Nmk J)\coloneq \delta_P^{-1/2}(m) \int_{K_M} f(ml)\chi_i(l^{-1}k)\, \dd l.\]
For each $f\in R$, 
\begin{align*} \sum_{i=1}^d \beta_i\alpha_i(f)(NmkJ) & = \sum_i \int_{K_M} f(Nml\kappa_i J)\chi_i(l^{-1}k)\, \dd l\\ 
& = f(Nmk J) \int_{K_M} \sum_i \chi_i(l^{-1}k)\, \dd l\\
&=f(NmkJ).\end{align*}
The second equality holds because $l^{-1}k\in K_N\kappa_i J \Rightarrow l\kappa_i\in K_N k J \Rightarrow Nml\kappa_i J=Nmk J$. The third equality holds because $\sum_i\chi_i$ is identically $1$ on $K$, and $\vol(K_M)=1$. 

The function $e\in\H(M)$ is supported on $K_M$, and satisfies $e*\chi_i=\chi_i$ for each $i$. It follows that 
\begin{equation}\label{beta_equation} \beta_i(e)(NmkJ) = \delta_P^{-1/2}(m)\int_{K_M} e(ml)\chi_i(l^{-1}k)\, \dd l = \begin{cases} \chi_i(mk)& \text{if }m\in K_M,\\ 0 & \text{otherwise.}\end{cases}\end{equation}
A straightforward computation combining \eqref{KMN_measure_equation} and \eqref{beta_equation} shows that
\[ [\alpha_j\pres(f)\beta_i](e)(m)=\delta_P^{1/2}(m)\vol_K(K_N\kappa_i J) \int_N f(\kappa_i^{-1} n m \kappa_j)\, \dd n\] for all $f\in C_c^\infty(J\backslash G/ J)$ and $m\in M$.

Finally, writing $v_i\coloneq \vol_K(K_N\kappa_i J)$, we have
\begin{align*}
&\delta_P^{-1/2}(m_0\cdots m_q)\Phi(f_0\otimes\cdots\otimes f_q)(m_0,\ldots,m_q) \\
&= \sum_{(i_0,\ldots, i_q)} (v_{i_0} \cdots v_{i_q})\, \int_{N^{q+1}} f_0(\kappa_{i_0}^{-1} n_0 m_0 \kappa_{i_1})\cdots f_q(\kappa_{i_q}^{-1} n_q m_q \kappa_{i_0})\, \dd n \\
& = \sum_{(i_0,\ldots, i_q)} [\alpha_{i_1} \pres(f_0) \beta_{i_0}](e)(m_0)\cdots [\alpha_{i_0}\pres(f_q) \beta_{i_q}](e)(m_q)\\
&= \delta_P^{-1/2}(m_0\cdots m_q) \phi(\pres(f_0)\otimes\cdots\otimes \pres(f_q))(m_0,\ldots,m_q)\end{align*} 
as required.
\end{proof}

Proposition \ref{Jacquet_proposition} allows us to interpret some results of Nistor \cite{Nistor} in terms of the Jacquet functor $\pres$. Recall that $\Class^\infty(G)$ denotes the algebra of locally constant functions on $G$, with pointwise multiplication. Restriction of functions gives a map $\Class^\infty(G)\to\Class^\infty(M)$, which we use to view $\h(\H(M))$ as a module over $\Class^\infty(G)$.

\begin{corollary}\label{Nistor_corollary1} (cf. \cite[Lemma 6.3]{Nistor}) The map $\pres:\h (\H(G))\to \h (\H(M))$ is $\Class^\infty(G)$-linear.\hfill\qed
\end{corollary}


In particular, considering the function $1_{G_c}\in \Class^\infty(G)$ we obtain:

\begin{corollary}\label{Jacquet_compact_corollary} $\pres  1_{G_c} = 1_{M_c} \pres$, as maps $\h (\H(G))\to \h (\H(M))$.\hfill\qed
\end{corollary}
In degree zero, this has also been observed by Dat \cite[Lemme 2.6]{Dat_Idempotents}.

For the group $G=\SL_2(F)$, and its diagonal subgroup $M$, Nistor has explicitly computed the kernel of the map $\Phi:\HH_*(\H(G))\to \HH_*(\H(M))$ \cite[Proposition 7.4]{Nistor}; among other things, it is shown that $\Phi$ is injective on $\HH_1(\H(G))$. Nistor also proves that the image of $\Phi$ is contained in the space of invariants of the Weyl group of $M$. These results, combined with Proposition \ref{Jacquet_proposition}, give: 

\begin{corollary}\label{pres_SL2_corollary} (cf. \cite[Proposition 7.4]{Nistor}) Let $G=\SL_2(F)$, and let $M$ be the subgroup of diagonal matrices. Let $\pind$ and $\pres$ denote the Jacquet functors with respect to the parabolic subgroup $P$ of upper-triangular matrices, and let $\opind$ and $\opres$ denote the Jacquet functors for the opposite parabolic $\overline{P}$ of lower-triangular matrices.
\begin{enumerate}
\item The map $\pres :\HH_1(\H(G))\to \HH_1(\H(M))$ is injective.
\item $\pres =\opres$ as maps $\h (\H(G))\to \h (\H(M))$.
\item $\pind=\opind$ as maps $\HH_1(\H(M))\to \HH_1(\H(G))$.
\end{enumerate}
\end{corollary}

\begin{proof} Part (1) is an immediate consequence of the cited result of Nistor and Proposition \ref{Jacquet_proposition}. Nistor's result and Proposition \ref{Jacquet_proposition} also show that $\Ad_w \pres = \pres$, where $w$ denotes the nontrivial element of the Weyl group of $M$. Since $\Ad_w\pres \cong \opres$ as functors, this proves part (2). For part (3), we use the geometric lemma to write
\[ \pres  \pind  = 1 + \Ad_w = \pres  \opind .\] In view of part (1), this shows that $\pind =\opind$ on $\HH_1$.
\end{proof}

The equality $\pind=\opind$ does not hold in general---for example, it doesn't hold on $\HH_0$ when $G=\SL_2$, see Proposition \ref{Iwahori_commutator_proposition}. We do not know whether the equality $\pres=\opres$ is valid for all reductive groups $G$ and Levi subgroups $M$; thanks to the referee for raising this question. We do  have the following partial result. Let $\h(\H(G))_{\reg}$ and $\h(\H(M))_{\reg}$ denote the localisations of the cyclic homology of $\H(G)$ and $\H(M)$ at the respective subsets of regular semisimple elements, as defined in \cite[Definition 3.4]{Blanc-Brylinski}. By \cite[Proposition 3.6]{Blanc-Brylinski}, one has embeddings $\h(\H(G))_{\reg}\into \h(\H(G))$ and $\h(\H(M))_{\reg}\into \h(\H(M))$.

\begin{proposition}\label{r_ropp_proposition}
Let $G$ be a reductive $p$-adic group, and let $M\subset G$ be a Levi factor of a parabolic subgroup $P$. The Jacquet restriction map $\pres$ maps $\h(\H(G))_{\reg}$ into  $\h(\H(M))_{\reg}$, and the restriction of $\pres$ to $\h(\H(G))_{\reg}$ does not depend on the choice of parabolic $P$.
\end{proposition}

\begin{proof}
We first note that Jacquet restriction maps $\h(\H(G))_{\reg}$ into $\h(\H(M))_{\reg}$: indeed, Proposition \ref{Jacquet_proposition} reduces this claim to the assertion that if $mn\in P=MN$ is a regular semisimple element of $G$, then $m$ is a regular semisimple element of $M$. This last statement is true, because \cite[Lemma 22 (i)]{HC-vD} implies that there is an element $\nu\in N$ such that $m=\nu mn\nu^{-1}$.

Now let $T$ be a maximal torus in $M$. We fix a Haar measure on $T$, and then specify invariant measures on $G/T$ and $M/T$ by $d_G=d_{G/T} d_T$ and $d_M=d_{M/T}d_T$. Let $I_T^G:\h (\H(G))_{\reg}\to \h (\H(T))_{\reg}$ and $I_T^M:\h (\H(M))_{\reg}\to \h (\H(T))_{\reg}$ be the \emph{higher orbital integrals} corresponding to these measures (see \cite[Proposition 4.2]{Blanc-Brylinski}). Define a function $D_{G/M}:T_{\reg}\to \R$ by
 \[ D_{G/M}(t)\coloneq \left |\det_{\mathfrak g/\mathfrak m}(\Ad(t)-\id) \right|\] where $\mathfrak g$ and $\mathfrak m$ are the respective Lie algebras of $G$ and $M$, and $|\argument|$ is the absolute value on $F^\times$. The function $D_{G/M}$ is locally constant on $T_{\reg}$, so it induces an endomorphism of $\h (\H(T))_{\reg}$. One shows that the diagram 
\[
 \xymatrix{
 \h(\H(G))_{\reg} \ar[r]^-{\pres} \ar[d]_-{I_T^G} & \h(\H(M)) \ar[d]^-{I_T^M} \\
 \h(\H(T))_{\reg} \ar[r]^-{D_{G/M}} & \h(\H(T))_{\reg}
 }
\]
 commutes, using the equality $\pres=\Phi$ from Proposition \ref{Jacquet_proposition}, a change of variables as in \cite[Lemma 9]{van_Dijk} and \cite[Lemma 5.5]{Kottwitz}, and the observation that $d_N(ntn^{-1}t^{-1})=|\det_{\mathfrak n}(\Ad(t)-\id)|d_N(n)$ for all $t\in T_{\reg}$ (see \cite[Lemma 22 (ii)]{HC-vD}). Neither $D_{G/M}$ nor $I_G^T$ depends on the choice of parabolic subgroup containing $M$, so the same is true of the map $I_T^M\circ \pres$. As $T$ ranges over the set of all maximal tori in $M$, the maps $I_T^M$ separate the points of $\h(\H(M))_{\reg}$, and so it follows that $\pres$ itself is independent of $P$.
\end{proof}

\subsection*{Compact restriction, central idempotents, and Clozel's formula}

Dat proves in \cite[Proposition 2.8]{Dat_Idempotents} that the compact-restriction operator $1_{G_c}$ commutes with the idempotents $E\in \Centre(G)$ on the degree-zero Hochschild homology of $\H(G)$, for every reductive $p$-adic group $G$. We conjecture that this commutation property holds also in higher homology:

\begin{conjecture}\label{commutation_conjecture} 
Let $G$ be a reductive $p$-adic group. Compact restriction $1_{G_c}$ commutes with all idempotents $E\in\Centre(G)$, as operators on the Hochschild and cyclic homology of $\H(G)$.
\end{conjecture}

Note that the commutator $[E,1_{G_c}]$ is certainly nilpotent, by Lemma \ref{SBI_lemma}. Conjecture \ref{commutation_conjecture} is true when $G$ is a torus, since in that case $1_{G_c}$ and $E$ already commute on $C(\H(G))$. The conjecture also holds in the next-simplest case:

\begin{theorem}\label{SL2_commutation_theorem} Let $G=\SL_2(F)$. Compact restriction $1_{G_c}$ commutes with all idempotents $E\in\Centre(G)$, as operators on the Hochschild and cyclic homology of $\H(G)$.
\end{theorem}

\begin{proof}
The Hochschild-to-cyclic spectral sequence (see \cite[2.1.7]{Loday}) allows us to deduce commutativity on cyclic homology from commutativity on Hochschild homology. Blanc and Brylinski prove in \cite[Section 6]{Blanc-Brylinski} that $\HH_n(\H(G))= 0$ for $n\geq 2$, and so given Dat's result for degree-zero homology we are left to prove the asserted commutation on $\HH_1(\H(G))$. Let $\pres$ denote Jacquet restriction to the diagonal subgroup $M\subset G$. For each idempotent $E\in \Centre(G)$, there is an idempotent $E_M\in \Centre(M)$ such that $\pres  E = E_M\pres$ as functors \cite[2.4]{Bernstein-Deligne-Kazhdan}. We have
\[ \pres[E,1_{G_c}] = [E_M, 1_{M_c}]\pres\] by Corollary \ref{Jacquet_compact_corollary}, and the commutator $[E_M,1_{M_c}]$ vanishes because $M$ is a torus. Since $\pres$ is injective in degree one (Corollary \ref{pres_SL2_corollary}), we conclude that $[E, 1_{G_c}]=0$ on $\HH_1(\H(G))$.
\end{proof}

The same proof applies to $\HH_n(\H(G))$ for any split group $G$ of rank $n$. We expect that an elaboration of this argument, using the higher orbital integrals and Shalika germs of Blanc-Brylinski \cite{Blanc-Brylinski} and Nistor \cite{Nistor}, will apply to the general case.

We shall now outline a different approach to Theorem \ref{SL2_commutation_theorem} and Conjecture \ref{commutation_conjecture}, following more closely the proof of Dat in degree zero. Once again let $G=\SL_2(F)$, and let $M$ be the diagonal subgroup. Let $\chi\in\Class^\infty(M)$ be the characteristic function of the set $\{m\in M\ |\ \delta_P(m)>1\}$. As above we write $\pres$ for Jacquet restriction from $G$ to $M$ along the upper-triangular subgroup $P$, and we write $\opind$ for parabolic induction from $M$ to $G$ along the lower-triangular subgroup $\overline{P}$.

\begin{theorem}\label{SL2_Clozel_theorem} The following is an equality of operators on the Hochschild and cyclic homology of $\H(G)$, for $G=\SL_2(F)$:
\[ 1_{G_c} + \opind \chi \pres = 1.\]
\end{theorem}

\begin{proof} In degree zero, this is Clozel's formula; see \cite{Clozel}, or below. As in the proof of Theorem \ref{SL2_commutation_theorem}, the results of \cite{Blanc-Brylinski} leave us to consider degree-one Hochschild homology. Corollary \ref{Jacquet_compact_corollary} and the geometric lemma imply that
\[ \pres\left( 1_{G_c} + \opind \chi \pres\right) = \left( 1_{M_c} + \chi + \Ad_w\chi\right)\pres.\] As endomorphisms of $C(\H(M))$, $\Ad_w\chi=\chi^w \Ad_w$, where $\chi^w$ is the function $m\mapsto \chi(m^{-1})$. Corollary \ref{pres_SL2_corollary} implies that $\Ad_w\pres=\opres=\pres$ in $\HH_1(\H(G))$, and so
\[ \pres\left(1_{G_c}+\opind\chi\pres\right)=\left(1_{M_c}+\chi +\chi^w\right)\pres.\]
The function $1_{M_c}+\chi+\chi^w$ is identically equal to $1$ on $M$, and so the proposed formula becomes a true equality upon applying $\pres$ to both sides. Since $\pres$ is injective on $\HH_1(\H(G))$, the formula itself holds. 
\end{proof}

We conjecture that Clozel's formula holds in higher homology for all reductive groups. More precisely, let $G$ be a reductive $p$-adic group, and choose a minimal parabolic subgroup $P_0\subset G$ and a Levi decomposition $P_0=L_0 U_0$. Then each parabolic subgroup $P$ containing $P_0$ has a unique Levi decomposition $P=LU$ with $L_0\subset L$. We write $L\leq G$ to indicate that $L$ is a Levi subgroup obtained in this way, and we let $\pres^G_L$ denote Jacquet restriction along $P$, and let $\opind_L^G$ denote parabolic induction along the opposite parabolic $\overline{P}$. For each $L\leq G$, let $A_L$ denote the split component of the centre of $L$. There is a positive integer $n$ such that, for all rational characters $\psi$ of $A_L$, the character $n\psi$ extends to $L$. Let $R_L= \{n\alpha\ |\ \alpha \textrm{ is a root of $A_L$ on $U$}\}$, and define
\[ L^+= \{l\in L\ |\ |\beta(l)|<1\text{ for all }\beta\in R_L\}.\] (See \cite[\S 0.5]{Silberger} or \cite[pp.239--240]{Clozel} for more details on $L^+$.) Let $L_{cz}\subset L$ denote the union of the compact-mod-centre subgroups of $L$, let $L_{cz}^+= L_{cz}\cap L^+$, and let $1_{L_{cz}^+}\in \Class^\infty(L)$ denote the characteristic function of $L_{cz}^+$. 

\begin{conjecture}\label{Clozel_conjecture} 
The following is an equality of operators on $\h (\H(G))$:
\[ \sum_{L\leq G} \opind_L^G 1_{L_{cz}^+} \pres^G_L  = 1.\]
\end{conjecture}

\begin{remarks}
\begin{enumerate}
\item As Dat points out in \cite[(2.1)]{Dat_Idempotents}, the conjectured formula in degree-zero homology is a reformulation of Clozel's formula \cite[Proposition 1]{Clozel} (and is therefore true). 
\item The conjectured equality holds on $\HH_n(\H(G))$ for every split group $G$ of rank $n$: this follows from results of \cite{Blanc-Brylinski} and the geometric lemma, as in the proof of Theorem \ref{SL2_Clozel_theorem}.
\item The validity of Conjecture \ref{Clozel_conjecture} would imply that of Conjecture \ref{commutation_conjecture}: Dat's proof of this fact in degree zero \cite[Proposition 2.8]{Dat_Idempotents} carries over verbatim to higher homology. 
\item Clozel's proof of \cite[Proposition 1]{Clozel} relies on a formula of Casselman \cite[Theorem 5.2]{Casselman_Jacquet} for the character of a Jacquet module, which amounts to a determination of the maps $\opind_L^G$ on degree-zero Hochschild homology. We expect that an explicit description of parabolic induction in higher homology will likewise yield a proof of Conjecture \ref{Clozel_conjecture}. 
\item Clozel's formula has a natural analogue for $G$ an affine Weyl group: see Section \ref{Weyl_section} for the statement and proof. 
\end{enumerate}
\end{remarks}

\subsection*{Compact restriction and parabolic induction for $\SL_2(F)$}

As an application of Theorem \ref{SL2_Clozel_theorem}, we compute the commutator of compact restriction and parabolic induction for $G=\SL_2(F)$ and its Levi subgroup $M$ of diagonal matrices. Dat uses Clozel's formula in \cite[Remarque, p.77]{Dat_Idempotents} to show that these operators do not commute in degree-zero homology. A similar argument, using Theorem \ref{SL2_Clozel_theorem}, shows that these operators do commute in higher homology, and that the failure to commute in degree zero is confined to a single Bernstein component: 

\begin{corollary}\label{SL2_induction_corollary} Let $M$ be the diagonal subgroup of $G=\SL_2(F)$. Using the notation of Theorem \ref{SL2_Clozel_theorem},
\begin{enumerate}
\item $1_{G_c}\pind - \pind 1_{M_c} = \left(\pind -\opind \right)\chi$ as maps $\h (\H(M))\to \h (\H(G))$.
\item $\pind \left(\h (\H(M))_c\right)\subseteq \h (\H(G))_c$.
\item $1_{G_c}\pind = \pind 1_{M_c}$ on $\HH_n(\H(M))$ and $\HC_n(\H(M))$ for every $n\geq 1$.
\item Let $E_1\in \Centre(M)$ be the minimal idempotent associated to the trivial character of $M^\circ$. Then $1_{G_c}\pind = \pind 1_{M_c}$ on $(1-E_1)\h(\H(M))$.
\end{enumerate}
\end{corollary}

(Note that part (2) implies that parabolic induction restricts to a map in chamber homology. This map is computed in \cite{Crisp_Parahoric}.) 

\begin{proof} Theorem \ref{SL2_Clozel_theorem} and the geometric lemma give
\[ 1_{G_c}\pind = (1-\opind \chi \pres)\pind = \pind  - \opind \chi(1+\Ad_w) = \pind(1-\chi^w) - \opind\chi.\]
Since $1_{M_c}+\chi +\chi^w$ is identically equal to $1$ on $M$, we find that
$1_{G_c}\pind - \pind 1_{M_c} = \pind \chi - \opind \chi$, giving (1). Multiplying (1) on the right by $1_{M_c}$ gives $1_{G_c}\pind 1_{M_c} - \pind 1_{M_c} = 0$, which implies (2). To prove part (3), we first note that part (1) and Corollary \ref{pres_SL2_corollary} give the asserted commutation on $\HH_1(\H(M))$. The vanishing on $\HH_n(\H(M))$ for $n\geq 2$ is trivial, because these Hochschild groups themselves are zero, and the vanishing on higher cyclic homology then follows from the exactness of the $\SBI$ sequence.  For part (4), we appeal to Kutzko's calculations in \cite{Kutzko_SL2}, which imply that $\pind=\opind$ on all Bernstein components except the unramified principal series. The function $\chi$ is $M^\circ$-invariant, so it commutes with $E_1$, and now (4) follows from (1).
%
\end{proof}

To complete the picture for $\SL_2(F)$, let us now compute the commutator $1_{G_c}\pind-\pind 1_{M_c}$ on $E_1\HH_0(\H(M))$. By \eqref{torus_equation}, $E_1\HH_0(\H(M))\cong \HH_0(\H(\Lambda))\cong \H(\Lambda)$, where $\Lambda=M/M^\circ\cong \Z$. Let $\lambda\in \Lambda$ be the generator satisfying $\chi(\lambda)=1$. We use the same symbol $\lambda$ to denote the function in $\H(\Lambda)$ taking the value $1$ on this generator, and zero elsewhere; so $\H(\Lambda)$ is isomorphic to the Laurent polynomial ring $\C[\lambda,\lambda^{-1}]$. 

The induction functors $\pind$ and $\opind$ send the Bernstein component $E_1\Mod_f(M)$ into the unramified principal-series component of $\Mod_f(G)$. The latter component is equivalent (cf. \cite{Borel_Iwahori},\cite{Casselman_Iwahori}) to the category $\Mod_f(\mathcal A)$ of finitely generated modules over the Iwahori-Hecke algebra 
$\mathcal A = \lspan_{\C}\{T_w\ |\ w\in W\}$ associated to the affine Weyl group $W=\langle s,t\ |\ s^2=t^2=1\rangle$. Recall that the multiplication in $\mathcal A$ is determined by the rules
\[ T_w T_w'= T_{ww'}\text{ if }\ell(ww')=\ell(w)+\ell(w'),\quad T_s^2=(q-1)T_s+q,\quad T_t^2=(q-1)T_t + q,\] where $\ell$ is the length function on $W$, and $q$ is the cardinality of the residue field of $F$. (We follow the notation of \cite{Kazhdan-Lusztig}; some of the references cited below use different conventions.)  

The degree-zero Hochschild homology group $\HH_0(\mathcal A)=\mathcal A/[\mathcal A,\mathcal A]$ has basis $\{ T_s, T_t, T_{(st)^n} | \ n\geq 0\}$. Computations in \cite{Kutzko_SL2} (a special case of the theory of types and covers \cite{Bushnell-Kutzko_Types}) lead to the following concrete identifications of the maps appearing in Clozel's formula. The induction maps $\pind$ and $\opind$ are the ones induced by the algebra homomorphisms
\[ \pind,\opind:\H(\Lambda)\to \mathcal A,\quad \pind(\lambda)= qT_{ts}^{-1}, \quad \opind(\lambda)= q^{-1}T_{st}.\]
The restriction map is
\[ \pres:\HH_0(\mathcal A) \to \H(\Lambda),\quad \pres(T_s)=\pres(T_t)=q-1,\quad \pres(T_{(st)^n})=q^n(\lambda^n+\lambda^{-n}).\]
The maps $1_{M_c}$ and $\chi$ are given on $\H(\Lambda)$ by
\[ 1_{M_c}(\lambda^n)=\begin{cases} 1&\text{if }n=0,\\ 0&\text{if }n\neq 0,\end{cases} \qquad \chi(\lambda^n)=\begin{cases}\lambda^n & \text{if }n\geq 1,\\ 0 & \text{if }n<1.\end{cases}\]
Putting these formulas into Clozel's shows that the compact restriction operator $1_{G_c}:\HH_0(\mathcal A)\to \HH_0(\mathcal A)$ is given by
\[ 1_{G_c}(T_s)=T_s,\quad 1_{G_c}(T_t)=T_t,\quad 1_{G_c}(T_{(st)^n})=\begin{cases} 1&\text{if }n=0,\\ 0 & \text{if }n\geq 1.\end{cases}\]

Finally, we compute the commutator of compact restriction and parabolic induction:
 
\begin{proposition}\label{Iwahori_commutator_proposition} Consider as above the unramified principal-series component of $G=\SL_2(F)$. The commutator $1_{G_c}\pind - \pind 1_{M_c}:\H(\Lambda)\to \HH_0(\mathcal A)$ is given by 
\[ (1_{G_c}\pind-\pind 1_{M_c})(\lambda^n)= \frac{R_{1,(st)^n}}{q^n(q-1)}\left(q-1-T_s-T_t\right)\] for $n\geq 1$, where $R_{1,(st)^n}=(q-1)(q^{2n-1}-q^{2n-2}+q^{2n-3}-\cdots -1)$.
The commutator vanishes on $\lambda^n$ for $n<1$.
\end{proposition}

\begin{proof} 
We know from Corollary \ref{SL2_induction_corollary} that 
\[ (1_{G_c}\pind - \pind 1_{M_c})(\lambda^n) = \left(\pind -\opind \right)\chi(\lambda^n)=\begin{cases} q^n T_{(ts)^{n}}^{-1} - q^{-n} T_{(st)^n}
&\text{if }n\geq 1,\\ 0 &\text{if }n<1.\end{cases}\] 
For $n<1$ there is nothing left to prove, so consider $n\geq 1$. We have
\begin{equation}\label{R_polynomial_equation} q^{n}T_{(ts)^{n}}^{-1}- q^{-n}T_{(st)^n} = q^{-n} \sum_{\ell(w) < 2n} (-1)^{\ell(w)} R_{w,(st)^n} T_w,\end{equation}
where the $R_{w,(st)^n}$ are certain polynomials in $q$: see \cite[\S 2]{Kazhdan-Lusztig}. The $R$-polynomials may be computed inductively, using the relations \cite[(2.0.a--c)]{Kazhdan-Lusztig}. Another inductive computation shows that the right-hand side of \eqref{R_polynomial_equation} reduces, modulo commutators, to $\frac{R_{1,(st)^n}}{q^n(q-1)}\left(q-1-T_s-T_t\right)$.     
\end{proof} 


%
%
%
%
%
%
%
%
%

Solleveld has computed the Hochschild and cyclic homology of arbitrary affine Hecke algebras \cite{Solleveld_Hochschild}. It would be interesting to extend the above computations to that general context.

\subsection*{Compact restriction and characters}

Let $M\subset G$ be a Levi subgroup of $G$, let $\Psi$ be the complex torus of unramified characters of $M$, let $M^\circ = \bigcap_{\psi\in \Psi}\ker\psi$, and let $\Lambda$ be the lattice $M/M^\circ$. For each irreducible supercuspidal representation  $\sigma$ of $M$, we consider the admissible $(G,\Psi)$-module $\pi\coloneq \pind_M^G\left(\mathcal O(\Psi)\otimes_{\C} \sigma\right)$. The character of $\pi$, as defined in Examples \ref{Keller_examples}(4), is a map $\ch_\pi:\h(\H(G))\to \h(\mathcal O(\Psi))$. The operator $1_{\Lambda_c}$ acts on $\h(\mathcal O(\Psi))$ via the Fourier isomorphism $\mathcal O(\Psi)\cong \H(\Lambda)$. Recall from \eqref{harmonic_equation} that in Hochschild homology, $1_{\Lambda_c}$ is the projection onto $\Psi$-invariant differential forms.

\begin{proposition}\phantomsection\label{character_proposition} 
\begin{enumerate}
\item One has $1_{\Lambda_c}\ch_\pi=\ch_\pi 1_{G_c}$ as maps $\HH_0(\H(G))\to \HH_0(\mathcal O(\Psi))$.
\item If Clozel's formula in higher homology (Conjecture \ref{Clozel_conjecture}) holds for $M$, then 
$1_{\Lambda_c}\ch_\pi = \ch_\pi 1_{G_c}$
as maps $\h(\H(G))\to \h(\mathcal O(\Psi))$.
\end{enumerate}
\end{proposition}


Since Clozel's formula holds for tori, part (2) applies to all principal-series characters. 

The proof of Proposition \ref{character_proposition} will require some preparation. Consider the diagram 
\[ \xymatrix@C=50pt{
\h(\H(G)) \ar[r]^-{\pres} \ar[d]_-{1_{G_c}} & \h(\H(M)) \ar[d]^-{1_{M_c}} \ar[r]^-{\ch_{\mathcal O(\Psi)\otimes\sigma}} & \h(\mathcal O(\Psi))\ar[d]^-{1_{\Lambda_c}}\\
\h(\H(G)) \ar[r]^-{\pres} & \h(\H(M)) \ar[r]^-{\ch_{\mathcal O(\Psi)\otimes\sigma}} & \h(\mathcal O(\Psi))
}\]
whose left-hand square commutes by Corollary \ref{Jacquet_compact_corollary}. Frobenius reciprocity implies that $\ch_\pi$ is equal to the composition of the horizontal arrows, and so we are left to show that the right-hand square commutes. Thus we may, and henceforth do, assume that $M=G$ and that $\pi=\mathcal O(\Psi)\otimes\sigma$.

Next, note that since $\Lambda$ is a quotient of $G$, $\Class^\infty(\Lambda)$ is a subalgebra of $\Class^\infty(G)$.

\begin{lemma}\label{character_linearity_lemma} The character $\ch_\pi:\h(\H(G))\to \h(\mathcal O(\Psi))$ is $\Class^\infty(\Lambda)$-linear.
\end{lemma}

\begin{proof} The space $C_n(\H(G))$ is spanned by elements of the form $f_0\otimes \cdots\otimes f_n$, where each $f_i$ is the characteristic function of a double coset $Hg_i H$ for some  $g_i\in G$ and compact open subgroup $H\subset G$. Fix an element of this form. The action of $f_i$  on $\pi^H = \mathcal O(\Psi)\otimes \sigma^H$ is given by
\[ \int_G f_i(g) g\otimes g \ \dd g = g_i\otimes \int_G f_i(g) g\ \dd g = \hat{g_i}\otimes f_i,\]
where $\hat{g_i}$ denotes pointwise multiplication by the function $\psi\mapsto \psi(g_i)$. Now, $\pi^H$ is a finite-rank free module over $\mathcal O(\Psi)$, and the trace map $\h(\End(\pi^H))\to \h(\mathcal O(\Psi))$ sends 
\[(\hat{g_0}\otimes f_0)\otimes \cdots \otimes (\hat{g_n}\otimes f_n) \mapsto \ch_\sigma(f_0\cdots f_n)\hat{g_0}\otimes \cdots\otimes \hat{g_n},\] where $\ch_\sigma:\H(G)\to \C$ is the usual character of the admissible representation $\sigma$. Thus,
\[ \ch_{\pi}(f_0\otimes\cdots\otimes f_n) = \ch_\sigma(f_0\cdots f_n) \hat{g_0}\otimes\cdots\otimes \hat{g_n}.\]

Now let $F\in \Class^\infty(\Lambda)$. Using the fact that $G^\circ$ contains all compact subgroups of $G$, one finds that $F(f_0\otimes\cdots\otimes f_n)= F(g_0\cdots g_n)f_0\otimes\cdots\otimes f_n$. On the other hand, it is immediate from the definition \eqref{Class_cyclic_equation} that $F(\hat{g_0}\otimes\cdots\otimes \hat{g_n})=F(g_0\cdots g_n)\hat{g_0}\otimes \cdots\otimes \hat{g_n}$. Thus $\ch_\pi$ commutes with $F$.
\end{proof}

\begin{proof}[Proof of Proposition \ref{character_proposition}]
Lemma \ref{character_linearity_lemma} implies that $1_{\Lambda_c}\ch_\pi = \ch_\pi 1_{G^\circ}$. Now, $G_c = G^\circ\cap G_{cz}$: one inclusion is immediate from the definitions, while the other follows from the fact that $G^\circ$ has compact centre. Therefore $1_{G_c}=1_{G^\circ}1_{G_{cz}}$ in $\Class^\infty(G)$. Applying Clozel's formula, we find that
\[ \ch_\pi - \ch_\pi 1_{G_{cz}} = \sum_{L\lneq G} \ch_\pi \opind^G_L 1_{L_{cz}^+} \pres^G_L = \sum_{L\lneq G} \ch_{\pres^G_L \pi} 1_{L_{cz}^+} \pres^G_L = 0,\] by the second adjoint theorem and the cuspidality of $\pi$. Therefore, assuming the validity of Clozel's formula, we have
\[ 1_{\Lambda_c}\ch_\pi = \ch_\pi 1_{G^\circ} = \ch_\pi 1_{G_{cz}} 1_{G^\circ} = \ch_\pi 1_{G_c}.\]
This proves part (2). Since Clozel's formula is certainly valid in degree zero for all $M$, part (1) follows.
\end{proof}

\begin{remark}
If the cuspidal datum $(M,\sigma)$ is generic (i.e., has trivial Weyl group), then the associated idempotent $E_{[M,\sigma]}\in \Centre(G)$ factors through $\ch_\pi$, and so the formula $1_{\Lambda_c}\ch_\pi = \ch_\pi 1_{G_c}$ gives a spectral description of the compact-restriction operator on the generic Bernstein components. In the spirit of Aubert, Baum, Plymen and Solleveld's conjectures on the geometric structure in the smooth dual of $G$ (\cite{ABP},\cite{ABPS}), one might ask whether this identification of the compact-restriction operator as a projection onto invariant differential forms is in fact valid on all Bernstein components. 
\end{remark}

\section{A ``Clozel formula'' for affine Weyl groups}\label{Weyl_section}

In this section we prove an analogue of Clozel's formula (\cite[Proposition 1]{Clozel}, cf. Conjecture \ref{Clozel_conjecture}) in the Hochschild and cyclic homology of affine Weyl groups. We begin with some generalities on induction and restriction maps for discrete groups.

Let $G$ be a discrete group, and let $\C(G)$ denote the complex group algebra. Every finite-index subgroup $L\subseteq G$ gives rise to a restriction functor $\pres_L:\Mod(G)\to \Mod(L)$ and an induction functor $\pind_L : \Mod(L)\to \Mod(G)$. Both functors preserve the subcategories of finitely generated projective modules, so they induce maps in Hochschild and cyclic homology: 
\[ \pind_L:\h(\C(L))\to \h(\C(G))\qquad \text{and}\qquad \pres_L:\h(\C(G))\to \h(\C(L)).\]
These maps are easy to describe explicitly: $\pind_L$ is the map induced by the obvious inclusion of cyclic modules $C(\C(L))\into C(\C(G))$, while $\pres_L$ is the trace map coming from a choice of $L$-equivariant isomorphism $\C(G)\cong \C(L)^{[G:L]}$. (The map in group homology corresponding to $\pres_L$ has been computed by Bentzen and Madsen \cite[Proposition 1.4]{Bentzen-Madsen}, and also by Blanc and Brylinski \cite[Proposition 8.2]{Blanc-Brylinski}.)

The algebra $\Class(G)$ of class functions on $G$ acts on the precyclic module $C(\C(G))$, according to the formula
\[ F(g_0\otimes g_1\otimes \cdots\otimes g_n)= F(g_0g_1\cdots g_n)g_0\otimes g_1\otimes\cdots\otimes g_n\] for $F\in \Class(G)$ and $g_i\in G$.
For each finite-index subgroup $L\subseteq G$, and each ${F}\in \Class(L)$, we let $\widetilde{F}\in \Class(G)$ be the function
\[ \widetilde{F}(g) = \sum_{k\in L\backslash G} F(kgk^{-1}),\] where $F(x)\coloneq 0$ if $x\not\in L$. 

\begin{lemma} \label{discrete_induction_lemma}
Let $L$ be a finite-index subgroup of a discrete group $G$. For each $F\in \Class(L)$ one has
$\widetilde{F} = \pind_L F \pres_L$ as operators on $\h(\C(G))$.
\end{lemma}

\begin{proof}
Choose a set $K\subset G$ of representatives for $L\backslash G$. On an elementary tensor $g_0\otimes\cdots\otimes g_n\in C_n(\C(G))$ one has
 \[ \begin{aligned} \pind_L F \pres_L&(g_0\otimes\cdots\otimes g_n) = \\ &  \sum_{\substack{k_0,\ldots,k_n\\ \in K^{n+1}}} \left(\begin{array}{c} 1_{L^{n+1}}(k_ng_0k_0^{-1},\ldots, k_{n-1}g_n k_n^{-1})\ F(k_ng_0\cdots g_n k_n^{-1})\ \times\\  k_ng_0k_0^{-1}\otimes\cdots\otimes k_{n-1}g_nk_n^{-1} 
 \end{array}
\right)
\end{aligned}
\]
where $1_{L^m}$ denotes the characteristic function of $L^{m}$ in $G^{m}$. For each $i=0,\ldots,n$, define $h_i:C_n(\C(G))\to C_{n-1}(\C(G))$ by
\[ \begin{aligned} h_i&(g_0\otimes\cdots\otimes g_n) = \\ & \sum_{\substack{k_0,\ldots,k_i\\ \in K^{i+1}}} \left(\begin{array}{c} 1_{L^i}(k_0 g_1 k_1^{-1},\ldots,k_{i-1}g_i k_i^{-1})\ F(k_i g_{i+1} \cdots g_n g_0\cdots g_i k_i^{-1})\ \times \\ g_0 k_0^{-1}\otimes k_0 g_1 k_1^{-1}\otimes\cdots \otimes k_{i-1} g_i k_i^{-1}\otimes k_i \otimes g_{i+1}\otimes\cdots \otimes g_n\end{array}
\right).
\end{aligned}
\]

It is a straightforward matter to verify that these maps $h_i$ constitute a presimplicial homotopy from $\widetilde{F}$ to $\pind_L F \pres_L$ (see, e.g., \cite[1.0.8]{Loday} for the terminology). We conclude that $\widetilde{F}=\pind_L F \pres_L$ on Hochschild homology. The Hochschild-to-cyclic spectral sequence \cite[2.1.7]{Loday} then gives the same equality on cyclic homology.
\end{proof}

Now let $V$ be a finite-dimensional Euclidean space, let $R\subset V$ be a (reduced) root system, let $W$ be the Weyl group of $R$, let $\Lambda\subset V$ be the lattice generated by $R$, and let $G= \Lambda\rtimes W$ be the affine Weyl group (of the coroot system $R^{\vee}$). We refer to \cite{Bourbaki_Lie} for basic facts and terminology regarding root systems.
Fix a basis $B$ for $R$. For each subset $S\subseteq B$, let $W_S$ be the subgroup of $W$ generated by the reflections $\{s_\alpha\ |\ \alpha\in S\}$, and let $G_S= \Lambda\rtimes W_S$. 
Notice that the centre of $G_S$ is the subgroup $\Lambda^{W_S}$ of $W_S$-invariants in $\Lambda$. For every $\psi\in \Hom(\Lambda^{W_S},\Z)$, the positive multiple $|W_S|\psi$ of $\psi$ extends uniquely to a homomorphism $\psi_S:G_S\to \Z$. In particular, the roots $\alpha\in R$ induce homomorphisms $\alpha_S:G_S\to \Z$.

\begin{definition}
For each $S\subseteq B$, let $G_S^+\subseteq G_S$ be the subset 
\[ G_S^+= \{g\in G_S\ |\ \alpha_S(g)>0\ \text{for all}\ \alpha\in  B\setminus S\},\]
let $(G_S)_{cz}$ be the set of elements of $G_S$ having finite order modulo the centre, and let $(G_S)_{cz}^+=(G_S)_{cz}\cap G_S^+$.
\end{definition}

For example, $(G_B)_{cz}^+$ is the set of torsion elements of $G$, while $(G_{\emptyset})_{cz}^+$ is the intersection of $\Lambda$ with the fundamental Weyl chamber determined by $B$.

Both $(G_S)_{cz}$ and $G_S^+$ are invariant under conjugation in $G_S$, so the characteristic function $1_{(G_S)_{cz}^+}$ acts on $\h(\C(G_S))$. Letting $\pind_{S}$ and $\pres_{S}$ denote the induction and restriction maps for the finite-index subgroup $G_S\subseteq G$, our ``Clozel formula'' for $G$ takes the following form: 

\begin{proposition}\label{Weyl_Clozel_proposition}
Let $G$ be the affine Weyl group associated to a root system $R$, and choose a basis $B\subset R$. Then $\{\pind_S 1_{(G_S)_{cz}^+} \pres_S\ |\ S\subseteq B\}$ is a set of pairwise-orthogonal, $\Class(G)$-linear, idempotent endomorphisms of $\h(\C(G))$, summing to the identity:    
\[ \sum_{S\subseteq B} \pind_{S} 1_{(G_S)_{cz}^+} \pres_{S} = 1.\]
\end{proposition}

We begin the proof with an alternative description of the sets $(G_S)_{cz}^+$. Let $\Lambda_S\subset \Lambda$ be the intersection of $\Lambda$ with the face of the fundamental Weyl chamber determined by $S$:
\[ \Lambda_S \coloneq \left(\bigcap_{\alpha\in S} \{\lambda \ |\ \langle \lambda,\alpha\rangle=0\} \right) \cap \left(\bigcap_{\alpha\in \Delta\setminus S} \{\lambda \ |\ \langle \lambda,\alpha\rangle >0\}\right).\]
An element $w\in W$ fixes a point in $\Lambda_S$ if and only if $w\in W_S$. One has a partition
\begin{equation}\label{Lambda_equation}
  \Lambda= \bigsqcup_{\substack{S\subseteq B,\\ k\in W/W_S}} k\Lambda_S.
\end{equation}
Notice that for each $g\in G$, the element $g^{|W|}$ lies in $\Lambda$.

\begin{lemma}\label{Weyl_Clozel_lemma}
For each subset $S\subseteq B$ one has $(G_S)_{cz}^+ = \{g\in G\ |\ g^{|W|}\in \Lambda_S\}$. 
\end{lemma}

\begin{proof}
We first claim that if $g^{|W|}\in \Lambda_S$, then $g\in G_S$. To show this, we write $g=\lambda w\in \Lambda\rtimes W$. Then $g^{|W|} = \frac{|W|}{n} (\lambda+w\cdot\lambda+\ldots+ w^{n-1}\cdot\lambda)$, where $n$ is the order of $w$ in $W$, and ``$\cdot$'' denotes the action of $W$ on $\Lambda$. It is then clear that $w\cdot g^{|W|}=g^{|W|}$, so $w$ fixes a point of $\Lambda_S$. This ensures that $w\in W_S$, so $g\in G_S$.

Suppose now that $g\in G_S$. It follows immediately from the definitions that $g\in (G_S)_{cz}^+$ if and only if $g^m\in (G_S)_{cz}^+$ for every $m\geq 1$. Taking $m=|W|$, it will thus suffice to prove that $\Lambda\cap (G_S)_{cz}^+=\Lambda_S$. The centre of $G_S$ is $\Lambda^{W_S}$, and the quotient $\Lambda/\Lambda^{W_S}$ is torsion-free (because $W_S$ acts linearly on $V$). It follows that 
\[\Lambda\cap (G_S)_{cz}=\Lambda^{W_S}=\{\lambda\in \Lambda\ |\ \langle \lambda,\alpha\rangle =0\text{ for all }\alpha\in S\}.\] 
Now, for $\lambda\in \Lambda^{W_S}$ and $\alpha\in  B$ we have $\alpha_S(\lambda)=|W_S|\langle \lambda,\alpha\rangle$, and so 
\[\Lambda^{W_S}\cap G_S^+=\{\lambda\in \Lambda^{W_S}\ |\  \langle \lambda,\alpha\rangle >0\text{ for all }\alpha\in  B\setminus S\}.\]
Combining the two displayed equalities gives $\Lambda\cap (G_S)_{cz}\cap G_S^+=\Lambda_S$ as required.
\end{proof}

\begin{proof}[Proof of Proposition \ref{Weyl_Clozel_proposition}]
The decomposition \eqref{Lambda_equation} implies that for each $g\in G$ there is a unique subset $S_g\subseteq  B$ and a unique $k\in W_{S_g}\backslash W$ such that $(kgk^{-1})^{|W|}\in \Lambda_{S_g}$. Applying Lemma \ref{Weyl_Clozel_lemma} gives, for each $S\subseteq B$,
\[ \sum_{k\in G_S\backslash G} 1_{(G_S)_{cz}^+}(kgk^{-1})=\begin{cases} 1&\text{if }S=S_g,\\ 0&\text{otherwise.}\end{cases}\]
Thus the functions $\widetilde{1_{(G_S)_{cz}^+}}$ constitute a pairwise-orthogonal family of idempotents in $\Class(G)$ summing to the identity. Lemma \ref{discrete_induction_lemma} implies that the same is true of the operators $\pind_S 1_{(G_S)_{cz}^+} \pres_S$.
\end{proof}

Proposition \ref{Weyl_Clozel_proposition} admits a straightforward generalisation to groups of the form $G' = G\times \Z^n$, where $G$ is an affine Weyl group. For each subset $S$ of some basis for the root system of $G$, one puts $G'_S\coloneq G_S\times \Z^n$ and $(G'_S)_{cz}^+\coloneq (G_S)_{cz}^+\times \Z^n$. Then $\sum_S \pind_S 1_{(G'_S)_{cz}^+} \pres_S=1$ on $\h(\C(G'))$, where $\pind_S$ and $\pres_S$ are the induction and restriction maps for the finite-index subgroups $G'_S\subseteq G'$; the proof is the same as above. Groups of this kind arise naturally in the study of reductive $p$-adic groups. For example, let ${T}$ be the diagonal torus in $p$-adic $\GL_n$, with normaliser ${N}$ and maximal compact subgroup ${T}_c$. The quotient $N/T_c$ is isomorphic to $G\times \Z$, where $G$ is the affine Weyl group of type $\widetilde{A_n}$.

\bibliography{compact_restriction}{}
\bibliographystyle{plain}

\end{document}